\documentclass[pdflatex,sn-mathphys-num]{sn-jnl}

\usepackage{amsmath,amssymb,amsthm,mathtools}
\usepackage{enumitem}
\usepackage{microtype}
\usepackage{booktabs}
\usepackage{xcolor}
\usepackage{comment}
\usepackage{url}
\providecommand{\doi}[1]{\url{https://doi.org/#1}}
\hypersetup{hidelinks,pdfauthor={Fahaar M. Pirani},pdftitle={DC decompositions and cutting-plane approximation: geometry and complexity}}

\theoremstyle{thmstyleone}
\newtheorem{theorem}{Theorem}[section]
\newtheorem{proposition}[theorem]{Proposition}
\newtheorem{lemma}[theorem]{Lemma}
\newtheorem{corollary}[theorem]{Corollary}
\theoremstyle{thmstyletwo}
\newtheorem{example}[theorem]{Example}
\theoremstyle{thmstyletwo}
\newtheorem{remark}[theorem]{Remark}
\newcounter{algorithm}

\newcommand{\R}{\mathbb{R}}

\newcommand{\Snp}{\mathbb{S}^{n}_{+}}
\newcommand{\range}{\operatorname{range}}
\newcommand{\rank}{\operatorname{rank}}
\newcommand{\diam}{\operatorname{diam}}
\newcommand{\conv}{\operatorname{conv}}
\newcommand{\argmin}{\operatorname*{argmin}}

\newcommand{\Bcal}{\mathcal{B}}
\newcommand{\Dcal}{\mathcal{D}}
\newcommand{\Wcal}{\mathcal{W}}

\begin{document}
\title[DC decompositions and cutting-plane approximation]{DC decompositions and cutting-plane approximation: geometry and complexity}
\author*[1]{\fnm{Fahaar M.} \sur{Pirani}}\email{fp268@cornell.edu}
\affil*[1]{\orgdiv{ORIE},
\orgname{Cornell University}, \orgaddress{\city{Ithaca},
\state{NY}, \country{USA}}}

\abstract{We study how Difference of Convex (DC) decomposition affects cutting-plane approximation in global optimization. For a smooth decomposition $f = g-h$, let $d$ be the dimension of the span of all Hessian ranges of $g$ over the convex hull of a compact feasible set. The smallest attainable dimension equals the minimum rank of a constant positive semidefinite matrix majorizing the Hessians of $f$ on that convex hull, and is attained by a quadratic first component. Retaining earlier gradient cuts and accepting new cuts only when the model error exceeds $\epsilon$ gives an $O(\epsilon^{-d/2})$ bound on accepted cuts. For an adaptive global DC method, this also bounds the number of subproblems solved exactly and globally. Under nondegeneracy, the exponent is sharp for uniform cutting-plane approximation of quadratic first components. A structured example reduces the certified exponent from $n/2$ to $1/2$ and the global subproblem to one variable.}


\maketitle

\section{Introduction}

Polyhedral approximation supplies lower bounds in global Difference of Convex (DC) optimization. For an objective $f=g-h$ with convex components, replacing
$g$ by the maximum $g_m$ of affine supporting cuts gives
$g_m-h\le f$. Global minimization of this lower model supplies a lower bound on the objective. A sufficiently small approximation error $\epsilon>0$ at its minimizer certifies approximate global optimality. The decomposition affects both the number of cuts needed and the structure of the global subproblems \cite{Tuy1995ORL,Tuy1995Handbook,LohneWagner2017,PiraniUlus2025}.
We study this dependence through the curvature of the approximated component.

A DC representation is never unique: a common convex function can be added to both components. The convex--concave procedure and the DC algorithm (DCA) linearize $h$, motivating decomposition criteria suited to that component \cite{AhmadiHall2018,LeThiPhamDinh2018,LeThiPhamDinh2024}.
For the global lower models considered here, the approximation error comes from $g$. Our criterion therefore concerns the directions in which $g$ must carry curvature. Throughout the paper, cuts refers to gradient supporting planes.

Let $X\subset\R^n$ be compact and full dimensional, and set
$\widehat X=\conv X$. For a $C^2$-smooth convex function $g$ on $\widehat X$, define
\begin{equation}\label{eq:VX}
 V_X(g):=\operatorname{span}\!\left(\bigcup_{z\in\widehat X}
       \range\nabla^2g(z)\right),\qquad d_X(g):=\dim V_X(g).
\end{equation}
The \emph{effective curvature dimension} $d_X(g)$ controls cutting-plane
error through projection onto $V_X(g)$. Directions orthogonal to this space
create no support error. The span of all Hessian ranges can exceed the
largest pointwise Hessian rank. So a pointwise rank calculation alone need not identify the approximate dimension.

Our first result characterizes the smallest attainable dimension over all
$C^2$-smooth DC decompositions whose components are convex on $\widehat X$:
\begin{equation}\label{eq:headline}
 d^\star(f;X)=\min\bigl\{\rank B:B\succeq0,\
 B-\nabla^2f(z)\succeq0\ \text{for every }z\in\widehat X\bigr\}.
\end{equation}
Thus an optimal dimension is attained by a quadratic first component,
without restricting the objective to a polynomial family. We give exact
conditions for curvature-free subspaces and explicit certificates for
uniformly concave directions and ridge-structured objectives. The
minimum-rank problem is nonconvex and semi-infinite. These characterizations do not supply a general efficient decomposition algorithm.

Finite-tolerance comparisons also depend on curvature scales and geometry of the feasible set. We prove that every admissible smooth decomposition has a quadratic
counterpart with the same isotropic dimension--curvature--diameter score. We also reduce the best packing bound obtained from a constant positive semidefinite Hessian majorant to the corresponding majorants of $f$. These reductions justify quadratic decomposition criteria for the stated bounds. They do not, however, assert equality of pointwise approximation errors.

The approximation result applies to a general cutting-plane refinement rule.
If models retain earlier cutting-plane inequalities and accept new cuts only
where model error exceeds $\epsilon>0$, their support points are separated in
$V_X(g)$. Packing then bounds accepted refinements by
$O(\epsilon^{-d_X(g)/2})$, independently of point selection.
For quadratic first components, uniform cutting-plane complexity is exactly
an internal covering number of $B^{1/2}X$, with exponent $\rank B/2$, when this set has nonempty relative interior in $\range B$. This sharpness concerns
uniform approximation, without imposing the same lower bound on an adaptive
optimization path.

One global optimization application is \cite[Algorithm 3]{PiraniUlus2025}, which minimizes successive lower models of $g-h$.
They prove finite termination of this adaptive method. Our general refinement theorem gives a tolerance-dependent bound on the number of global subproblems solved by their algorithm. The implementation in
\cite{PiraniUlus2025} uses the low-dimensional solver of \cite{LohneWagner2017}. The cost of solving each global subproblem remains separate from our count. Nevertheless, Example~\ref{ex:one} shows a concrete benefit of
structure: the certified exponent falls from $n/2$ to $1/2$, and the global subproblem reduces to one-dimensional problem. Negation exchanges $g$ and $h$ and gives the corresponding analysis for maximization of $f$.

Outer approximation and covering methods for global DC optimization are classical \cite{Tuy1995ORL,Tuy1995Handbook,BlanqueroCarrizosa2000,
BigiFrangioniZhang2010,FerrerBagirovBeliakov2015,LohneWagner2017}.
Quadratic convexification and semidefinite majorization are also established results
\cite{AdjimanEtAl1998,SkjalEtAl2012,SkjalWesterlund2014,
FernandezFritzen2020,StrahlEtAl2025,Zlobec2003,Zlobec2006}.
Unlike optimization over polynomial families for the convex--concave procedure
\cite{AhmadiHall2018}, our criterion ranges over all admissible $C^2$-smooth convex decompositions on $\widehat X$. Thus, our contribution is the resulting
connection between optimal DC curvature allocation, finite-tolerance
quadratic reductions, and decomposition-sensitive cutting-plane bounds.

\section{Smooth DC curvature allocation}

For the smooth results throughout the paper, $X\subset\R^n$ is compact and full dimensional,
but need not be convex. Set $\widehat X=\conv X$ and assume that $f$ is $C^2$-smooth on an open neighborhood of $\widehat X$.
Let $\Dcal^2(f;X)$ be the set of pairs $(g,h)$ that are $C^2$-smooth on a neighborhood of $\widehat X$, satisfy $f=g-h$ there, and are convex on $\widehat X$.
Equivalently,
\[
       \nabla^2g(z)\succeq0,\qquad \nabla^2h(z)\succeq0
       \quad(z\in\widehat X).
\]
The set $\Dcal^2(f;X)$ is nonempty.  Compactness and continuity imply that
$\beta I\succeq\nabla^2f(z)$ on $\widehat X$ for all sufficiently large
$\beta$.  Then $g(x)=\tfrac\beta2\|x\|^2$ and $h(x)=g(x)-f(x)$ form an admissible pair.
Define
\[
       d^\star(f;X):=\min_{(g,h)\in\Dcal^2(f;X)}d_X(g).
\]
The minimum is well defined because the attained dimensions form a nonempty
subset of $\{0,\ldots,n\}$.

For $(g,h)\in\Dcal^2(f;X)$, set
\[
 L_X(g):=\sup_{z\in\widehat X}\lambda_{\max}(\nabla^2g(z)),
 \qquad
 D_X(g):=\diam(P_{V_X(g)}X),
\]
where $P_V$ denotes the Euclidean orthogonal projection onto $V$, and define the
Bregman support error
\[
      \beta_g(y,x):=g(y)-g(x)-\langle\nabla g(x),y-x\rangle.
\]
If $d_X(g)=0$, $g$ is affine on $\widehat X$ and $L_X(g)=D_X(g)=0$.
Full dimensionality also gives $D_X(g)=0\Rightarrow d_X(g)=0$.

\begin{lemma}%[Support-error geometry]
\label{lem:support}
For all $x,y\in X$,
\[
      \beta_g(y,x)\le \frac{L_X(g)}2
        \|P_{V_X(g)}(y-x)\|^2.
\]
\end{lemma}
\begin{proof}
Let $d=y-x$.  Taylor's formula along $[x,y]\subset\widehat X$ gives
\[
 \beta_g(y,x)=\int_0^1(1-t)d^\top\nabla^2g(x+td)d\,dt.
\]
Every Hessian range is contained in $V_X(g)$ and therefore vanishes on $V_X(g)^\perp$. The integrand is therefore at most $L_X(g)\|P_{V_X(g)}d\|^2$.  Integration proves the claim.
\end{proof}

\section{Semidefinite and subspace characterizations}

For a $C^2$-smooth function $\varphi$, let
\[
 \Bcal(\varphi;X):=
 \{B\in\Snp:B-\nabla^2\varphi(z)\succeq0
                         \text{ for every }z\in\widehat X\}.
\]
Every $B\in\Bcal(f;X)$ gives the quadratic decomposition
\begin{equation}\label{eq:quadratic-decomposition}
   g_B(x)=\frac12x^\top Bx,
   \qquad h_B(x)=g_B(x)-f(x).
\end{equation}

\begin{theorem}[Minimum-rank characterization]\label{thm:minrank}
The optimal smooth DC curvature dimension satisfies
\[
          d^\star(f;X)=\min_{B\in\Bcal(f;X)}\rank B.
\]
In particular, an optimal effective curvature dimension is attained by a
quadratic first component.
\end{theorem}

\begin{proof}
Fix an admissible pair $(g,h)$, and denote $V=V_X(g)$ and
$M=L_X(g)$.  Since $\nabla^2g(z)\succeq0$ and
$\range\nabla^2g(z)\subset V$,
\[
       0\preceq\nabla^2g(z)\preceq MP_V.
\]
Moreover, $\nabla^2g(z)-\nabla^2f(z)=\nabla^2h(z)\succeq0$. Hence
$MP_V\in\Bcal(f;X)$.  
If $M=0$, positive semidefiniteness forces every Hessian
of $g$ to vanish and $V=\{0\}$. Otherwise $$\range(MP_V)=V.$$ Thus in both cases
$\rank(MP_V)=d_X(g)$, proving that the minimum on the right-hand side is at most
$d^\star(f;X)$.

Conversely, each $B\in\Bcal(f;X)$ yields the admissible decomposition
\eqref{eq:quadratic-decomposition}, with
$V_X(g_B)=\range B$ and $d_X(g_B)=\rank B$.
\end{proof}

Negation swaps the components: $$(g,h)\in\Dcal^2(f;X) \text{ if and only if }
(h,g)\in\Dcal^2(-f;X).$$ Thus Theorem~\ref{thm:minrank} also gives
\begin{equation}\label{eq:negative-minrank}
\begin{aligned}
 d^\star(-f;X)
 &=\min_{(g,h)\in\Dcal^2(f;X)}d_X(h)\\
 &=\min\{\rank B:B\succeq0,\ B+\nabla^2f(z)\succeq0\ (z\in\widehat X)\}.
\end{aligned}
\end{equation}
However, the optimal curvature dimensions for $f$ and $-f$ need not agree.

\begin{proposition}%[Minimal feasible curvature subspace]
\label{prop:subspace}
One has
\[
 d^\star(f;X)=\min\bigl\{\dim V:\exists M\ge0\ \text{such that}\
        MP_V-\nabla^2f(z)\succeq0\ \text{for every }z\in\widehat X\bigr\}.
\]
Whenever $(V,M)$ is feasible,
$g(x)=\tfrac M2\|P_Vx\|^2$ and $h(x)=g(x)-f(x)$ gives a decomposition whose first
component has curvature confined to $V$.
\end{proposition}

\begin{proof}
A feasible pair $(V,M)$ gives $MP_V\in\Bcal(f;X)$. Applying Theorem~\ref{thm:minrank} yields
$$d^\star(f;X)\le\dim V.$$  Conversely, if $B\in\Bcal(f;X)$ and $V=\range B$, then $B\preceq\lambda_{\max}(B)P_V$, and thus
$$\lambda_{\max}(B)P_V\in\Bcal(f;X).$$ Minimizing over $B$ proves the result.
\end{proof}

\begin{proposition}[Affine invariance]\label{prop:affine}
Let $A\in\R^{n\times n}$ be invertible, $b\in\R^n$,
$Y=A^{-1}(X-b)$, and $F(y)=f(Ay+b)$.  Then
\[
                 d^\star(F;Y)=d^\star(f;X).
\]
\end{proposition}

\begin{proof}
Since $\nabla^2F(y)=A^\top\nabla^2f(Ay+b)A$, every
$B\in\Bcal(f;X)$ induces $A^\top BA\in\Bcal(F;Y)$ with unchanged rank. Applying Theorem~\ref{thm:minrank} gives us one inequality. The reverse inequality follows by applying the same argument to the inverse transformation.
\end{proof}

\section{Exact characterization of curvature-free subspaces}

In this section, we provide exact conditions to characterize a curvature-free subspace and certificates for structured classes. To that end, fix a subspace $W\subset\R^n$ and set $V=W^\perp$.  Relative to
$\R^n=V\oplus W$, write
\begin{equation}\label{eq:block-hessian}
 \nabla^2f(z)=
 \begin{pmatrix}
     H_{VV}(z)&H_{VW}(z)\\
     H_{WV}(z)&H_{WW}(z)
 \end{pmatrix}.
\end{equation}
We call $W$ a \emph{feasible curvature-free subspace} if some
$B\in\Bcal(f;X)$ satisfies $W\subset\ker B$.  Since $B\succeq0$, such a matrix
has the form $B=\operatorname{diag}(B_V,0)$ with $B_V\succeq0$.  Whenever
$-H_{WW}(z)\succeq0$, define
\begin{equation}\label{eq:RW}
 R_W(z):=H_{VV}(z)+H_{VW}(z)(-H_{WW}(z))^\dagger H_{WV}(z),
\end{equation}
where $A^\dagger$ denotes the Moore--Penrose pseudoinverse of $A$.  If $V=\{0\}$,
conditions involving $R_W$ are vacuous.

\begin{theorem}[Feasible curvature-free subspaces]\label{thm:schur}
A subspace $W$ is feasible curvature-free if and only if, for every
$z\in\widehat X$,
\begin{align}
       H_{WW}(z)&\preceq0, \label{eq:nonpositive}\\
       \range H_{WV}(z)&\subset\range(-H_{WW}(z)), \label{eq:range}
\end{align}
and
\begin{equation}\label{eq:bounded-R}
       \sup_{z\in\widehat X}\lambda_{\max}(R_W(z))<+\infty.
\end{equation}
When these conditions hold and $V\ne\{0\}$, $B=\alpha P_V$ is feasible for every
\[
       \alpha\ge
       \max\left\{0,\sup_{z\in\widehat X}\lambda_{\max}(R_W(z))\right\}.
\]
If $V=\{0\}$, feasibility is equivalent to $\nabla^2f(z)\preceq0$
throughout $\widehat X$, and $B=0$ is feasible.
More generally, for $B_V\succeq0$, one has
$\operatorname{diag}(B_V,0)\in\Bcal(f;X)$ if and only if
\eqref{eq:nonpositive}--\eqref{eq:range} hold and
$B_V\succeq R_W(z)$ for every $z\in\widehat X$.
\end{theorem}

\begin{proof}
Feasibility is equivalent, for every $z$, to
\[
 \begin{pmatrix}
   B_V-H_{VV}(z)&-H_{VW}(z)\\
   -H_{WV}(z)&-H_{WW}(z)
 \end{pmatrix}\succeq0.
\]
The classical generalized Schur-complement criterion \cite{Zhang2005} says that a symmetric block matrix
$\left(\begin{smallmatrix}A&E\\E^\top&C\end{smallmatrix}\right)$ is positive
semidefinite if and only if $C\succeq0$, $\range E^\top\subset\range C$, and
$A-EC^\dagger E^\top\succeq0$.  Applying this criterion gives
\eqref{eq:nonpositive}, \eqref{eq:range}, and $B_V\succeq R_W(z)$.  A common
$B_V$ exists precisely when the family $R_W(z)$ has a finite common upper bound.
In finite dimensions this is equivalent to \eqref{eq:bounded-R}, because
$R_W(z)\preceq\alpha I_V$ for the displayed choice of $\alpha$.
\end{proof}

\begin{remark}
Compactness does not ensure \eqref{eq:bounded-R}: the pseudoinverse can become unbounded ner points where $\rank(-H_{WW}(z))$ drops, despite the pointwise
range condition.
\end{remark}

\begin{corollary}%[Uniform negativity]
\label{cor:uniform-neg}
If $H_{WW}(z)\preceq-\mu I_W$ for all $z\in\widehat X$ and some $\mu>0$, then
$W$ is feasible curvature-free.
\end{corollary}

\begin{proof}
The range condition is automatic and
$\|(-H_{WW}(z))^{-1}\|_{\mathrm{op}}\le1/\mu$.  Continuity of the Hessian blocks
and compactness of $\widehat X$ then bound $R_W$ uniformly above. Applying Theorem~\ref{thm:schur} gives the desired result.
\end{proof}

\begin{corollary}[Vanishing cross-curvature]\label{cor:crosszero}
If $H_{VW}(z)=0$ and $H_{WW}(z)\preceq0$ for all $z\in\widehat X$, then $W$ is
feasible curvature-free.
\end{corollary}

\begin{proof}
The range condition is trivial and $R_W(z)=H_{VV}(z)$, which is uniformly
bounded above by continuity and compactness.
\end{proof}

Let $\Wcal(f;X)$ be the nonempty family of subspaces satisfying
\eqref{eq:nonpositive}--\eqref{eq:bounded-R}.  It is nonempty because
$W=\{0\}$ is always admissible.

\begin{corollary}[Exact kernel formula]\label{cor:kernel}
\[
             d^\star(f;X)=n-\max_{W\in\Wcal(f;X)}\dim W.
\]
\end{corollary}

\begin{proof}
If $B\in\Bcal(f;X)$, Theorem~\ref{thm:schur} gives
$\ker B\in\Wcal(f;X)$ and $\rank B=n-\dim\ker B$.  Conversely,
Theorem~\ref{thm:schur} constructs a feasible majorant with any
$W\in\Wcal(f;X)$ in its kernel.  Applying Theorem~\ref{thm:minrank} gives the desired result.
\end{proof}

\begin{example}[Cross-curvature obstruction]\label{ex:cross}
Let $f(x,y)=xy$ and $W=\operatorname{span}\{e_2\}$.  Then $H_{WW}=0$, so the
Hessian quadratic form is nonpositive on $W$.  However, $H_{WV}=1$, and
\[
  \range H_{WV}=\R\not\subset\{0\}=\range(-H_{WW}).
\]
Thus \eqref{eq:range} fails.  Equivalently, every positive semidefinite $B$ with
$W\subset\ker B$ has the form $\left(\begin{smallmatrix}b&0\\0&0\end{smallmatrix}\right)$,
while $B-\nabla^2f$ has determinant $-1$.
\end{example}

\section{Geometry and structured objectives}

Define the largest common nonpositive subspace dimension
\[
 m_-(f;X):=\max\bigl\{\dim W:w^\top\nabla^2f(z)w\le0
      \ \text{for all }w\in W,\ z\in\widehat X\bigr\}.
\]
Let $n_+(Q)$ denote the number of positive eigenvalues of a symmetric matrix
$Q$, counted with multiplicity.

\begin{proposition}[Joint and pointwise lower bounds]\label{prop:lower}
\[
 d^\star(f;X)\ge n-m_-(f;X)
        \ge \sup_{z\in\widehat X}n_+(\nabla^2f(z)).
\]
\end{proposition}

\begin{proof}
If $B\in\Bcal(f;X)$ and $w\in\ker B$, then
$w^\top\nabla^2f(z)w\le0$ for every $z$.  Hence
$$\dim\ker B\le m_-(f;X) \text{ and } \rank B\ge n-m_-(f;X).$$  At a fixed $z$, a
subspace on which $\nabla^2f(z)$ is nonpositive has dimension at most
$n-n_+(\nabla^2f(z))$. Applying Theorem~\ref{thm:minrank} gives the desired result.
\end{proof}

\begin{theorem}[A quantitative sufficient condition]\label{thm:quantitative}
Suppose a subspace $W$ and $\mu>0$ satisfy
\[
 w^\top\nabla^2f(z)w\le-\mu\|w\|^2
       \quad(w\in W,\ z\in\widehat X).
\]
Let $V=W^\perp$ and
$R=\sup_{z\in\widehat X}\|\nabla^2f(z)\|_{\mathrm{op}}$.  Then
\[
       \left(R+\frac{R^2}{\mu}\right)P_V\in\Bcal(f;X),
       \qquad d^\star(f;X)\le n-\dim W.
\]
\end{theorem}

\begin{proof}
Write $u=v+w$ with $v\in V$ and $w\in W$.  For every $z\in\widehat X$,
\[
 u^\top\nabla^2f(z)u
 \le R\|v\|^2+2R\|v\|\|w\|-\mu\|w\|^2
 \le\left(R+\frac{R^2}{\mu}\right)\|v\|^2,
\]
where the last step uses Young's inequality.  Applying Proposition~\ref{prop:subspace}
gives the desired result.
\end{proof}

\begin{proposition}[Ridge-structured objectives]\label{prop:ridge}
Let $A\in\R^{r\times n}$ have rank $r$, let $\varphi$ be $C^2$-smooth on a neighborhood of $A\widehat X$, and let
\[
         f(x)=\varphi(Ax)+q(x),
\]
where $q$ is concave quadratic.  If
$\nabla^2\varphi(u)\preceq LI_r$ on $A\widehat X$ for some $L\ge0$, then
$d^\star(f;X)\le r$.  If, in addition, $\nabla^2f(z)$ has at least $r$ positive
eigenvalues at some $z\in\widehat X$, then $d^\star(f;X)=r$.
\end{proposition}

\begin{proof}
Concavity of $q$ yields $\nabla^2f(x)\preceq LA^\top A$.  Thus
$LA^\top A\in\Bcal(f;X)$ and has rank at most $r$.  The additional assumption and
Proposition~\ref{prop:lower} give equality.
\end{proof}

\begin{comment}

\end{comment}

\begin{example}[Rotating positive curvature]\label{ex:rotating}
Let $f(x,y)=x^3-3xy^2$ on the unit ball in $\R^2$.  Then
\[
 \nabla^2f(x,y)=6\begin{pmatrix}x&-y\\-y&-x\end{pmatrix},
\]
whose eigenvalues away from the origin are
$\pm6\sqrt{x^2+y^2}$.  Thus
$\sup_Xn_+(\nabla^2f)=1$, but $d^\star(f;X)=2$.  Indeed, if a feasible $B$ had
rank at most one and $0\ne w=(a,b)\in\ker B$, feasibility at $(x,y)$ and
$(-x,-y)$ would force
\[
 w^\top\nabla^2f(x,y)w=6[x(a^2-b^2)-2yab]=0
\]
throughout the disk.  Hence $a^2-b^2=ab=0$, a contradiction. Conversely,
$6I\in\Bcal(f;X)$. Geometrically, the positive eigenspace rotates through all directions as $(x,y)$ varies throughout the ball. The example shows why pointwise positive inertia may underestimate the joint curvature dimension.
\end{example}

\section{Finite accuracy quadratic reductions}

For $\epsilon>0$, define the isotropic score
\[
 C_\epsilon(g;X):=
 \begin{cases}
 \left(1+D_X(g)\sqrt{2L_X(g)/\epsilon}\right)^{d_X(g)},
       &d_X(g)>0,\\
 1,&d_X(g)=0.
 \end{cases}
\]
For $B\succeq0$, let $V_B=\range B$, $D_B(X)=\diam(P_{V_B}X)$, and define
$C_\epsilon(B;X)$ by replacing $d_X(g)$, $L_X(g)$, and $D_X(g)$ by
$\rank B$, $\lambda_{\max}(B)$, and $D_B(X)$, respectively. The score records
the dimensional exponent and a simple Euclidean packing constant, but does not retain the full anisotropic geometry of $B$.

\begin{theorem}[Quadratic reduction of the isotropic score]\label{thm:score}
For every $\epsilon>0$,
\[
 \inf_{(g,h)\in\Dcal^2(f;X)}C_\epsilon(g;X)
   =\inf_{B\in\Bcal(f;X)}C_\epsilon(B;X).
\]
More precisely, every admissible smooth decomposition has an associated
quadratic decomposition with the same score.
\end{theorem}

\begin{proof}
Given $(g,h)\in\Dcal^2(f;X)$, set
$B=L_X(g)P_{V_X(g)}$ as in Theorem~\ref{thm:minrank}.  Then
\[
 \rank B=d_X(g),\quad \lambda_{\max}(B)=L_X(g),\quad
 \range B=V_X(g),\quad D_B(X)=D_X(g).
\]
The zero-curvature case is covered by $B=0$.  Conversely, every
$B\in\Bcal(f;X)$ induces \eqref{eq:quadratic-decomposition}.  Equality of the
scores does not assert equality of pointwise support errors.
\end{proof}

For $B\succeq0$, write $\|z\|_B=(z^\top Bz)^{1/2}$.  If
$B\succeq\nabla^2g(z)$ on $\widehat X$, then
\begin{equation}\label{eq:anisotropic-error}
          \beta_g(y,x)\le\frac12\|y-x\|_B^2
          \qquad(x,y\in\widehat X).
\end{equation}
For a compact Euclidean set $S$, let $\mathcal P(S,\rho)$ be the supremal
cardinality of a subset whose distinct points are more than $\rho$ apart.

\begin{theorem}%[Best constant semidefinite packing geometry]
\label{thm:packing-geometry}
For every $\epsilon>0$,
\[
 \inf_{(g,h)\in\Dcal^2(f;X)}\ \inf_{B\in\Bcal(g;X)}
 \mathcal P(B^{1/2}X,\sqrt{2\epsilon})
 =\inf_{B\in\Bcal(f;X)}\mathcal P(B^{1/2}X,\sqrt{2\epsilon}).
\]
Moreover, for every $B\in\Bcal(f;X)$,
\[
 \mathcal P(B^{1/2}X,\sqrt{2\epsilon})
 \le\left(1+\frac{\sqrt2\,\diam(B^{1/2}X)}{\sqrt\epsilon}\right)^{\rank B}.
\]
\end{theorem}

\begin{proof}
If $B\in\Bcal(g;X)$, then
$B\succeq\nabla^2g=\nabla^2f+\nabla^2h\succeq\nabla^2f$, so
$B\in\Bcal(f;X)$.  

Conversely, every $B\in\Bcal(f;X)$ generates $g_B$ with
$B\in\Bcal(g_B;X)$. The volume estimate is the standard packing bound in the Euclidean space $\range B$, whose dimension is $\rank B$.
\end{proof}

\begin{remark}
A minimum rank majorant need not minimize the packing number at a fixed $\epsilon$. Eigenvalues
and the geometry of $B^{1/2}X$ also matter.
\end{remark}

\section{Adaptive cutting-plane refinement and global DC optimization}

We first bound retained gradient cuts of a smooth convex function. Global DC optimization will follow as a corollary. Let $g$ be $C^2$-smooth on a
neighborhood of $\widehat X$ and convex on $\widehat X$. %with the notationof Section~2. 
At a support point $x_i\in X$, define the gradient cutting plane
\[
 \ell_i(z):=g(x_i)+\langle\nabla g(x_i),z-x_i\rangle.
\]
Fix $\epsilon>0$ and one initial support point $x_0$. Let $m_j$ denote the model of $g$ after $j$ accepted refinements. The standard choice is $m_j = \max_{0 \le i \le j} \ell_i$ with $m_0 = \ell_0$. For $j \ge 1$, a refinement at $x_j$ is \emph{accepted} when the current model $m_{j-1}$ satisfies
\begin{equation}\label{eq:retained-model}
 m_{j-1}(z)\ge\ell_i(z)\quad(z\in X,\ 0\le i<j),
 \qquad g(x_j)-m_{j-1}(x_j)>\epsilon.
\end{equation}
The indices $j$ count accepted refinements only. Other models are allowed if they
continue to dominate every earlier cutting-plane inequality.

\begin{lemma}%[Separation of accepted support points]
\label{lem:separation}
At every accepted refinement $j$,
\[
 \beta_g(x_j,x_i)>\epsilon\qquad(0\le i<j).
\]
\end{lemma}
\begin{proof}
For every $i<j$, \eqref{eq:retained-model} gives
$$\beta_g(x_j,x_i)=g(x_j)-\ell_i(x_j)
\ge g(x_j)-m_{j-1}(x_j)>\epsilon.$$
\end{proof}

\begin{theorem}[General cutting-plane refinement bound]\label{thm:refinement}
Under \eqref{eq:retained-model}, no refinement is accepted if $d_X(g)=0$.
If $d_X(g)>0$, the number $q$ of accepted new cuts satisfies
\begin{equation}\label{eq:Nepsilon}
 q+1\le N_\epsilon:=
 \mathcal P\left(P_{V_X(g)}X,\sqrt{\frac{2\epsilon}{L_X(g)}}\right)
 \le\left(1+D_X(g)\sqrt{\frac{2L_X(g)}{\epsilon}}\right)^{d_X(g)}.
\end{equation}
Moreover, every $B\in\Bcal(g;X)$ gives the anisotropic bound
\[
 q+1\le\mathcal P(B^{1/2}X,\sqrt{2\epsilon}).
\]
These bounds hold independently of how the candidate points and models are
computed.
\end{theorem}
\begin{proof}
If $d_X(g)=0$, $g$ is affine on $\widehat X$ and $\ell_0=g$ there, precluding
\eqref{eq:retained-model}. Otherwise set $V=V_X(g)$ and
$\rho=\sqrt{2\epsilon/L_X(g)}$. Lemmas~\ref{lem:support} and
\ref{lem:separation} give
\[
 \epsilon<\beta_g(x_j,x_i)
 \le\frac{L_X(g)}2\|P_V(x_j-x_i)\|^2\qquad(i<j).
\]
Thus $P_Vx_0,\ldots,P_Vx_q$ are pairwise more than $\rho$ apart.
Their number is at most $\mathcal P(P_VX,\rho)$, and the Euclidean volume
bound gives the second inequality. Applying this argument to each finite
prefix also rules out infinitely many accepted refinements.
The anisotropic assertion follows in the
same way from \eqref{eq:anisotropic-error}.
\end{proof}

\begin{remark}%[Scope and cut retention]
Theorem~\ref{thm:refinement} bounds the number of accepted refinements. Each subproblem is assumed to be solved exactly. The theorem does not bound arithmetic operations, oracle calls inside a global solver, or the complexity of inexact subproblem solution.
\end{remark}

We now apply the theorem to DC optimization. Let $h$ be finite, continuous,
and convex on $\widehat X$, and set
$$f^\star=\min_{x\in X}\{g(x)-h(x)\}.$$
The following is the smooth version of the lower-model mechanism in
\cite[Algorithm~3]{PiraniUlus2025}, with each subproblem solved globally.
Convexity outside $\widehat X$ is not required, so all decompositions in
$\Dcal^2(f;X)$ are included.

\begin{center}
\begin{minipage}{0.94\textwidth}
\hrule\vspace{0.4em}
\refstepcounter{algorithm}\textbf{Algorithm \thealgorithm: Adaptive global lower-model method}\label{alg:adaptive}
\begin{enumerate}[leftmargin=2em,itemsep=0.15em]
\item Choose $x_0\in X$ and set $g_0=\ell_0$.
\item For $k=1,2,\ldots$, compute globally
\[
 x_k\in\argmin_{x\in X}\{g_{k-1}(x)-h(x)\}.
\]
\item If $g(x_k)-g_{k-1}(x_k)\le\epsilon$, return $x_k$.
Otherwise set $g_k=\max\{g_{k-1},\ell_k\}$ and repeat.
\end{enumerate}
\vspace{0.2em}\hrule
\end{minipage}
\end{center}

\begin{corollary}[DC outer iterations and global certification]\label{thm:algorithm}
Algorithm~\ref{alg:adaptive} returns an $\epsilon$-optimal global solution
after solving at most $N_\epsilon$ lower-model subproblems globally when
$d_X(g)>0$, and after the first subproblem when $d_X(g)=0$.
Every $B\in\Bcal(g;X)$ also bounds the number of solved subproblems by
$\mathcal P(B^{1/2}X,\sqrt{2\epsilon})$.
Consequently, the outer-iteration count is $O(\epsilon^{-d_X(g)/2})$.
\end{corollary}
\begin{proof}
Each nonterminating loop accepts a cut satisfying
\eqref{eq:retained-model}. Therefore, Theorem~\ref{thm:refinement} forces termination.
When $d_X(g)>0$, if $K$ is the terminating loop, $q=K-1$ cuts were accepted and
$K=q+1\le N_\epsilon$ subproblems were solved. Note that the initialization does not count as a solved subproblem.
At termination, $g_{K-1}\le g$ and global solution of the lower model gives
\[
 g(x_K)-h(x_K)
 \le\min_{x\in X}\{g_{K-1}(x)-h(x)\}+\epsilon
 \le f^\star+\epsilon.
\]
The anisotropic and zero-curvature conclusions follow from the same theorem.
\end{proof}

The count is independent of the exact solver used for the lower-model
subproblems. It does not bound the work inside that solver. In
Example~\ref{ex:one}, the box constraints and separable second component also reduce
subproblem dimension; effective curvature dimension alone does not guarantee
such a reduction.

\begin{remark}[Maximization by sign reversal]\label{rem:sign-reversal}
For a smooth pair $(g,h)\in\Dcal^2(f;X)$, apply Algorithm~\ref{alg:adaptive}
to $-f=h-g$.
It builds a cutting-plane lower model $h_{k-1}$ of $h$ and globally minimizes
$h_{k-1}-g$. Corollary~\ref{thm:algorithm} then gives
\[
 f(x_K)\ge\max_{x\in X}f(x)-\epsilon,
 \qquad K\le C_\epsilon(h;X).
\]
The bound uses $d_X(h),L_X(h),D_X(h)$. If $d_X(h) = 0$ the algorithm solves only a single subproblem. This certifies maximization of $f$. The original algorithm certifies minimization.
\end{remark}

\begin{example}[A single effective curvature direction]\label{ex:one}
Let $a\in\R^n$ be a unit vector, $\mu>0$, and
\[
 f(x)=(a^\top x)^4-\frac\mu2\|P_{a^\perp}x\|^2.
\]
The decomposition $g(x)=(a^\top x)^4$ and
$h(x)=\tfrac\mu2\|P_{a^\perp}x\|^2$ has $d_X(g)=1$.  Since
\[
 \nabla^2f(z)=12(a^\top z)^2aa^\top-\mu P_{a^\perp}
\]
and a full-dimensional $\widehat X$ is not contained in $a^\perp$, the Hessian
has one positive eigenvalue at some $z\in\widehat X$. Hence
$d^\star(f;X)=1$. 

Now, consider $-f$. The swapped first component, which is $h$, has dimension $n-1$, matching the positive inertia of $-\nabla^2f$. Hence $d^\star(-f;X)=n-1$.

To compare decompositions explicitly, let $a=e_1$ and
$X=[-1,1]^n$. The structured first component $g(x)=x_1^4$ has
\[
 d_X(g)=1,\qquad L_X(g)=12,\qquad D_X(g)=2.
\]
An isotropic quadratic convexification of the same objective is
\[
 g_{\mathrm{iso}}(x)=6\|x\|^2,\qquad
 h_{\mathrm{iso}}(x)=g_{\mathrm{iso}}(x)-f(x).
\]
Indeed, the Hessian of $h_{\mathrm{iso}}$ is diagonal, with entries
$12-12x_1^2$ and $12+\mu$ in the remaining directions, so both components
are convex on $X$. This representation has
\[
 d_X(g_{\mathrm{iso}})=n,\qquad L_X(g_{\mathrm{iso}})=12,\qquad
 D_X(g_{\mathrm{iso}})=2\sqrt n.
\]
The scores defined in Section~6 are therefore
\[
 \begin{aligned}
 C_\epsilon(g;X)&=1+4\sqrt{6/\epsilon},\\
 C_\epsilon(g_{\mathrm{iso}};X)&=
       \left(1+4\sqrt{6n/\epsilon}\right)^n.
 \end{aligned}
\]
Thus the certified outer-iteration exponent changes from $n/2$ to $1/2$
when the curvature assigned to the approximated component is confined to
the first coordinate.

For this particular box example, the structured decomposition also simplifies
the lower-model subproblem. Gradient cuts of $g$ depend only on $x_1$, so
the model in Algorithm~\ref{alg:adaptive} satisfies
\[
 \min_{x\in[-1,1]^n}\{g_{k-1}(x)-h(x)\}
 =\min_{t\in[-1,1]}g_{k-1}(t,0,\ldots,0)-\frac\mu2(n-1).
\]
Choose the other coordinates at $\pm1$, leaving a one-variable convex
polyhedral problem. This reduction uses the box and separable $h$.
\end{example}

\section{Sharpness for uniform cutting-plane approximation}

For a compact set $S$ in a Euclidean space and $r>0$, define the internal
covering number
\[
 \mathcal N(S,r):=\min\left\{m:\exists s_1,\ldots,s_m\in S,\
                S\subset\bigcup_{i=1}^m(s_i+r\mathbb B)\right\}.
\]

\begin{theorem}[Exact covering formula]\label{thm:covering}
Let $B\succeq0$, $g_B(x)=\tfrac12x^\top Bx$, and let $X$ be compact.  For
$a_1,\ldots,a_m\in X$, define
\[
 g_m(x):=\max_{1\le i\le m}
 \{g_B(a_i)+\langle Ba_i,x-a_i\rangle\}.
\]
Then
\[
 \sup_{x\in X}(g_B(x)-g_m(x))
 =\frac12\sup_{x\in X}\min_{1\le i\le m}
           \|B^{1/2}(x-a_i)\|^2.
\]
Consequently, the minimum number of cuts of this form needed for uniform error
at most $\epsilon$ is exactly
$\mathcal N(B^{1/2}X,\sqrt{2\epsilon})$.
\end{theorem}

\begin{proof}
For each $a_i$,
\[
 g_B(x)-g_B(a_i)-\langle Ba_i,x-a_i\rangle
       =\frac12\|B^{1/2}(x-a_i)\|^2.
\]
Take the minimum over $i$ and then the supremum over $x\in X$. The covering
centers are internal because $B^{1/2}a_i\in B^{1/2}X$. Conversely, every center
in $B^{1/2}X$ has a preimage in $X$, even if $B$ is singular.
\end{proof}

\begin{theorem}[Sharp effective-dimension exponent]\label{thm:sharp}
Let $B\succeq0$ have rank $d>0$, set $S=B^{1/2}X\subset\range B$, and suppose
that for some $s_0\in\range B$ and $0<r_-\le r_+$,
\[
       s_0+r_-\mathbb B_d\subset S\subset s_0+r_+\mathbb B_d,
\]
where the balls are taken in $\range B$.  Let $m_\epsilon$ be the minimum number
of gradient cuts based in $X$ needed for uniform approximation of $g_B$ within
$\epsilon$, and set $\delta_\epsilon=\sqrt{2\epsilon}$.  If
$\delta_\epsilon<r_-$, then
\[
    \left(\frac{r_-}{\delta_\epsilon}\right)^d
    \le m_\epsilon\le
    \left(1+\frac{2r_+}{\delta_\epsilon}\right)^d.
\]
In particular, $m_\epsilon=\Theta(\epsilon^{-d/2})$.
\end{theorem}

\begin{proof}
By Theorem~\ref{thm:covering},
$m_\epsilon=\mathcal N(S,\delta_\epsilon)$.  Any such cover must cover the
$d$-dimensional ball of radius $r_-$ contained in $S$. Therefore, comparison of volumes gives the lower bound.  For the upper bound, take a maximal
$\delta_\epsilon$-separated subset of $S$.  Maximality makes it an internal
$\delta_\epsilon$-cover.  Its balls of radius $\delta_\epsilon/2$ are disjoint
and lie in a ball of radius $r_++\delta_\epsilon/2$, giving the upper bound.
\end{proof}

\begin{corollary}\label{cor:ball}
For $X=R\mathbb B_d$ and $g(x)=\tfrac\mu2\|x\|^2$, with $R,\mu>0$, the minimum
number of gradient cuts required for uniform error at most $\epsilon$ is
\[
       \Theta\left(\left(R\sqrt{\frac\mu\epsilon}\right)^d\right)
       =\Theta(\epsilon^{-d/2}).
\]
\end{corollary}

\begin{remark}
Uniform sharpness does not give a pathwise lower bound for
Algorithm~\ref{alg:adaptive}, which may terminate before uniform error
reaches $\epsilon$. If $B^{1/2}X$ has empty relative interior in
$\range B$, the uniform exponent may be smaller than $\rank(B)/2$.
\end{remark}

\section{DC dominance and quadratic objectives}

Let $(g_1,h_1),(g_2,h_2)\in\Dcal^2(f;X)$.  Following the classical dominance
order, the first pair \emph{weakly dominates} the second if
$q=g_2-g_1=h_2-h_1$ is convex on $\widehat X$
\cite{BomzeLocatelli2004,AhmadiHall2018}.

\begin{theorem}\label{thm:dominance}
Suppose $(g_1,h_1)$ weakly dominates $(g_2,h_2)$, so that
$q=g_2-g_1$ is convex on $\widehat X$. Then
\[
               V_X(g_2)=V_X(g_1)+V_X(q).
\]
Consequently,
\[
 \begin{aligned}
 d_X(g_1)&\le d_X(g_2),&
 L_X(g_1)&\le L_X(g_2),\\
 D_X(g_1)&\le D_X(g_2),&
 C_\epsilon(g_1;X)&\le C_\epsilon(g_2;X).
 \end{aligned}
\]
For any common support set in $\widehat X$, the complete lower model generated
by $(g_1,h_1)$ is pointwise no smaller on $\widehat X$ than the model generated
by $(g_2,h_2)$.
\end{theorem}

\begin{proof}
For positive semidefinite matrices $A$ and $C$,
$\range(A+C)=\range A+\range C$.  Apply this identity to
$\nabla^2g_2(z)=\nabla^2g_1(z)+\nabla^2q(z)$ and take spans over
$z\in\widehat X$.  The four inequalities follow from subspace inclusion,
L\"owner monotonicity, and the fact that the bases in the definition of
$C_\epsilon$ are at least one.

At common support points, write the gradient cut of $g_2$ as
$\ell_i^{g_1}+\ell_i^q$.  Convexity gives $\ell_i^q\le q$, and hence
\[
 \max_i(\ell_i^{g_1}+\ell_i^q)-h_1-q
       \le \max_i\ell_i^{g_1}-h_1.
\]
This compares the complete lower models on common supports, without
ordering executions that generate different support points.
\end{proof}

\begin{theorem}[Quadratic objectives]\label{thm:quadratic}
Let
\[
      f(x)=\frac12x^\top Qx+c^\top x+\beta,
      \qquad Q=Q^\top.
\]
Then $d^\star(f;X)=n_+(Q)$.  If $Q=Q_+-Q_-$ is the positive--negative spectral
decomposition, the pair
\[
      g(x)=\frac12x^\top Q_+x,
      \qquad
      h(x)=\frac12x^\top Q_-x-c^\top x-\beta
\]
attains the minimum.
\end{theorem}

\begin{proof}
If $B\in\Bcal(f;X)$ and $w\in\ker B$, then $w^\top Qw\le0$.  Therefore
$\dim\ker B\le n-n_+(Q)$ and $\rank B\ge n_+(Q)$.  The choice $B=Q_+$ is
feasible because $Q_+-Q=Q_-\succeq0$.
\end{proof}

\begin{comment}
    \section{Conclusion}

Smooth DC decomposition choice controls the certified complexity of
cutting-plane approximation through effective curvature dimension and the
geometry of Hessian majorants. The smallest dimension has a quadratic
realization, while the finite-tolerance reductions account for curvature
scales and feasible-set geometry. The subspace conditions give explicit
certificates when the objective has suitable concave or ridge structure.

The general refinement theorem bounds accepted gradient cuts under retention
of earlier inequalities. In the global DC application, it bounds exact
lower-model solves needed for an approximate global solution; negation gives
the corresponding maximization result. Uniform quadratic approximation
supplies the sharp dimensional exponent under the stated nondegeneracy
condition. The structured box example also
reduces global subproblem dimension, using properties beyond curvature
dimension alone. Constructing useful low-rank majorants efficiently and
extending the global analysis to inexact subproblem solves remain open
directions.
\end{comment}

\section*{Declarations}

% Replace the two pending statements below with confirmed author declarations
% before submission. No funding or competing-interest facts have been assumed.

\subsection*{Data availability}
This manuscript has no associated data.

\subsection*{Funding}
The author received no specific funding for this work.

\subsection*{Competing interests}
The author declares no relevant financial or non-financial competing interests.

\subsection*{AI acknowledgment}
Example~\ref{ex:cross} was found by a couple of prompts to ChatGPT 5.5 on High mode in April 2026. This discovery led the author to derive equivalent conditions to characterize curvature-free subspaces. Moreover, AI was used to polish some parts of this paper, mainly the introduction. The author assumes responsibility for all content.

% Keep each reference entry together across page breaks.

\bibliography{JOGO_DC_curvature}

\end{document}